\documentclass[12pt, leqno, british, a4paper]{amsart}
\usepackage[
	style=alphabetic,
	citestyle=alphabetic,
	backend=biber
]{biblatex}
\usepackage{a4, amsmath}
\usepackage{mathtools}
\usepackage{amssymb}
\usepackage{graphics}
\usepackage{amsthm, amscd, mathdots}
\swapnumbers
\usepackage{enumerate}
\usepackage{url}
\usepackage{cancel}
\usepackage{hyperref}
\usepackage{cleveref}
\usepackage{csquotes}
\usepackage{color}
\usepackage{datetime}
\usepackage{todonotes}

\usepackage[pass]{geometry}

\usepackage[all]{xy}

\theoremstyle{definition}
\newtheorem{defi}{Definition}[section]

\theoremstyle{plain}
\newtheorem{prop}[defi]{Proposition}
\newtheorem{lem}[defi]{Lemma}
\newtheorem{stel}[defi]{Theorem}
\newtheorem{gev}[defi]{Corollary}

\newtheorem{ques}[defi]{Question}
\newtheorem*{stel*}{Theorem}
\newtheorem*{prop*}{Proposition}
\newtheorem*{gev*}{Corollary}
\theoremstyle{remark}
\newtheorem{opm}[defi]{Remark}

\newcommand{\La}{\mathcal{L}}
\newcommand{\nat}{\mathbb{N}}
\newcommand{\zz}{\mathbb{Z}}
\newcommand{\rr}{\mathbb{R}}
\newcommand{\ff}{\mathbb{F}}

\newcommand{\qq}{\mathbb{Q}}

\newcommand{\mc}{\mathcal}
\newcommand{\mf}{\mathfrak}
\newcommand{\mbb}{\mathbb}
\newcommand{\pow}[1]{^{(#1)}}

\newcommand{\Lar}{{\mathcal{L}_{\rm ring}}}
\DeclareMathOperator{\subs}{\vert}

\DeclareMathOperator{\charac}{char}

\DeclareMathOperator{\Frac}{Frac}

\DeclareMathOperator{\Trd}{Trd}

\DeclareMathOperator{\Nrd}{Nrd}

\DeclareMathOperator{\Odd}{\mathsf{Odd}}

\newcommand{\fie}{\mathsf{field}}
\newcommand{\ovl}{\overline}

\newcommand{\sep}[3][]{
\ifx &#1&
{{#2}_{(#3)}}
\else
{{#2}_{(#3, #1)}}
\fi
}

\title{Universally defining $\zz$ in $\qq$ with $10$ quantifiers}
\author{Nicolas Daans}
\date{\today}
\thanks{This is the accepted version of the following article, which has been published in final form at \href{https://doi.org/10.1112/jlms.12864}{https://doi.org/10.1112/jlms.12864}: \\
Nicolas Daans. “Universally defining $\zz$ in $\qq$ with $10$ quantifiers”. In: \textit{Journal of the London Mathematical Society} 109.2 (2024), e12864
}
\address{Universiteit Antwerpen, Departement Wiskunde, Middelheimlaan 1, 2020 Antwerpen, Belgium.}
\address{Charles University, Faculty of Mathematics and Physics, Department of Algebra, Sokolov\-sk\' a 83, 186~75 Praha~8, Czech Republic.}
\email{nicolas.daans@matfyz.cuni.cz}

\begin{document}
\begin{abstract}
We show that for a global field $K$, every ring of $S$-integers has a universal first-order definition in $K$ with $10$ quantifiers.
We also give a proof that every finite intersection of valuation rings of $K$ has an existential first-order definition in $K$ with $3$ quantifiers.
\end{abstract}
\maketitle

\section{Introduction}
It is a longstanding open problem whether the ring of integers $\zz$ has an \emph{existential first-order definition} in the field of rational numbers $\qq$ in the signature of rings.
In more algebraic terms, the question is whether there exist a natural number $m$ and a polynomial $F \in \qq[X, Y_1, \ldots, Y_m]$ such that
$$ \zz = \lbrace x \in \qq \mid \exists y_1, \ldots, y_m \in \qq : F(x, y_1, \ldots, y_m) = 0 \rbrace.$$
While the answer to this question still eludes us, Koenigsmann was able to show that the complement $\qq \setminus \zz$ is existentially definable in $\qq$ \autocite{Koe16}.
In other words, he showed that there exist a natural number $m$ and a polynomial $F \in \qq[X, Y_1, \ldots, Y_m]$ such that
\begin{equation}\label{eq:ZVinQ}
\zz = \lbrace x \in \qq \mid \forall y_1, \ldots, y_m \in \qq : F(x, y_1, \ldots, y_m) \neq 0 \rbrace.
\end{equation}
One also says that $\zz$ has a \emph{universal first-order definition} in $\qq$, and the number $m$ is called the \emph{number of quantifiers}.
In this note, we show that one can find a polynomial $F$ such that \eqref{eq:ZVinQ} holds already for $m = 10$, i.e.~$\zz$ has a universal first-order definition in $\qq$ with $10$ quantifiers.

In fact, we show something more. Work by Park \autocite{Par13}, Eisentr{\"a}ger and Morrison \cite{Eis18} and the author \autocite{DaansGlobal} revealed that Koenigsmann's method can be applied more generally to show that in any global field $K$, any ring of $S$-integers has a universal first-order definition.
By a \emph{global field} we mean either a \emph{number field}, i.e.~a finite field extension of $\qq$, or a \emph{global function field}, i.e.~ a function field in one variable over a finite field.
For a global field $K$ and a finite (possibly empty) set $S$ of valuations on $K$, the \emph{ring of $S$-integers} is defined to be the intersection of all valuation rings of $K$ except those which are given by valuations in $S$.
Observe that $\zz$ is the ring of $\emptyset$-integers of $\qq$.
Our main result can be summarised as follows.
\begin{stel*}[see \Cref{T:nicomainthmQO}]
Let $K$ be a global field, $S$ a finite set of valuations on $K$. There exists a polynomial $F \in K[X, Y_1, \ldots, Y_{10}]$ such that, for the ring of $S$-integers $\mc{O}_S$, we have
\begin{displaymath}
\mc{O}_S = \lbrace x \in K \mid \forall y_1, \ldots y_{10} \in K : F(x, y_1, \ldots, y_{10}) \neq 0 \rbrace.
\end{displaymath}
\end{stel*}
In \autocite{Koe16, Par13, Eis18} the number of quantifiers was not counted; according to a preprint of Koenigsmann's article, his technique leads to a universal definition with $418$ quantifiers \autocite[Theorem 1]{Koe16unpub}.
In \autocite{DaansGlobal} it was shown that rings of $S$-integers in global fields have a universal definition with $37$ quantifiers; in the case of $\zz$ in $\qq$, this was further refined by Sun and Zhang to $32$ quantifiers in \autocite{ZhangSunZinQ}.

The study of the number of quantifiers needed to existentially define subsets of fields is motivated for several reasons.
For example, it is well-known that, if $\zz$ would be existentially definable in $\qq$, then it would follow that there is no algorithm which decides whether or not a polynomial equation has a zero in $\qq$.
This observation can be made quantitative (see \autocite[Proposition 8.21, Remark 8.22]{DDF}): if $\zz$ would be existentially definable in $\qq$ with $N$ quantifiers for some natural number $N$, then it would follow that every recursively enumerable subset of $\qq$ would be existentially definable with $12N$ quantifiers.
In particular, it would then follow from the negative solution to Hilbert's 10th Problem that there is no algorithm which decides whether or not a polynomial equation in $12N$ variables has a zero over $\qq$.

We can use a similar argument to deduce the following undecidability result from the universal definability of $\zz$ in $\qq$ with $10$ quantifiers:

\begin{gev*}[see \Cref{C:undecidable}]
There exists $F \in \zz[X, Y_1, \ldots, Y_{9}, Z_1, \ldots, Z_{10}]$ with the following property.
There is no algorithm which decides, for a given $x \in \qq$, whether or not
$$ \forall y_1, \ldots, y_{9} \in \qq\enspace \exists z_1, \ldots, z_{10} \in \qq : F(x, y_1, \ldots, y_{9}, z_1, \ldots, z_{10}) = 0. $$
\end{gev*}
Said informally, the above corollary says that the \emph{$\forall_9\exists_{10}$-theory of $\qq$ is undecidable}.
Koenigsmann already showed that the $\forall\exists$-theory of $\qq$ is undecidable (i.e.~without counting the number of universal or existential quantifiers).
The undecidability of the $\forall_9\exists_{32}$-theory of $\qq$ was observed in \autocite[Theorem 1.3]{ZhangSunZinQ}.

On the way to the proof of our main theorem, we further obtain some other classical existential definability results with better bounds.
For example:
\begin{prop*}[see \Cref{P:valE3}]
Let $K$ be a global field, $R$ a finite intersection of valuation rings of $K$.
Then there exists a polynomial $F \in K[X, Y_1, Y_2, Y_3]$ such that
$$ R = \lbrace x \in K \mid \exists y_1, y_2, y_3 \in K : F(x, y_1, y_2, y_3) = 0 \rbrace.$$
\end{prop*}
The fact that valuation rings (and hence also finite intersections of valuation rings) of global fields are existentially definable has been known for decades, and in certain cases this was even already shown to be possible with $3$ quantifiers; see the discussion in \Cref{R:valE}. We provide a conceptually lean argument for the above proposition which covers all cases at once, and which furthermore does so uniformly, see \Cref{R:uniformity}.

This paper is structured as follows.
The following two sections contain preliminaries.
More precisely, in \Cref{sect:Edef} some general (mostly well-known) results are stated on existentially definable subsets over fields, in particular global fields, with special attention given to the number of quantifiers.
Quaternion algebras (and quadratic forms) over global fields have played a historic role in establishing existential definability of subrings of global fields.
Hence, in \Cref{sect:Equat} we survey some algebraic ingredients regarding global fields and quaternion algebras - more details can be found in \autocite[Sections 3 and 4]{DaansGlobal}.

In \Cref{sect:definingValuationsGlobal} we state and prove the announced existential definability result for valuation rings in global fields.
We also develop some techniques to existentially define certain subsets of valuation rings and cartesian products of valuation rings with fewer quantifiers than one would do naively.
\Cref{sect:ZinQQO} contains the proof of the main theorem.
Finally, in the shorter \Cref{sect:recursive}, we discuss the implications of this result on the study of recursively enumerable subsets of $\qq$.

\subsection*{Acknowledgements}
The author thanks Yong Hu for pointing out an error in a previous version of this manuscript in the beginning of \Cref{sect:definingValuationsGlobal}, Silvain Rideau-Kikuchi for pointing out some ambiguities in a previous version of the proof of \Cref{T:nicomainthmQO}, and the anonymous referee for multiple suggestions which helped improve the presentation of the article.

This work grew out of the author's PhD dissertation \autocite{DaansThesis}, which was supported by the FWO PhD Fellowship fundamental research grants 51581 and 83494.

\section{Existentially definable subsets of fields and number of quantifiers}\label{sect:Edef}
There are different ways to define what it means for a subset of a field to be existentially definable, and it can be convenient to switch between these equivalent definitions depending on the context.
These equivalences are well-known, but are often proven without the goal in mind of keeping the number of quantifiers low.
As such, in this section, we provide short proofs or references for these statements with quantitative bounds.
We conclude the section with some general results about the number of quantifiers for existentially definable subsets of global fields.

We denote by $\nat$ the set of natural numbers, and by $\nat^+$ the proper subset of nonzero natural numbers.

We will use the basic set-up of first-order languages, as covered by many introductory textbooks on logic or model theory, see e.g.~\autocite[Chapter II-III]{Ebb94}.
We denote by $\Lar$ the \emph{signature of rings}.
It consists of two constant symbols $0$ and $1$, and three binary operation symbols $+$, $-$ and $\cdot$.
Similarly, $\La_{\fie}$ denotes the \emph{signature of fields}, which consists of two constant symbols $0$ and $1$, three binary operation symbols $+, -$ and $\cdot$, and a unary operation symbol $.^{-1}$.
Given a field $K$, we interpret $K$ as an $\Lar$-structure or as an $\La_{\fie}$-structure in the natural way; we take the convention that $0^{-1} = 0$.
When $C \subseteq K$, we denote by $\Lar(C)$ the signature obtained by adding to $\Lar$ a constant symbol for every element of $C$, and we can then interpret $K$ as an $\Lar(C)$-structure in the natural way.

For a signature $\La$, an $\La$-formula $\varphi$, a variable $X$ and an $\La$-term $t$, we write $\varphi(X \subs t)$ for the formula obtained by substituting all freely occurring instances of $X$ in $\varphi$ by $t$.
When introducing a first-order formula $\varphi$ in a signature $\La$, we might write $\varphi(X_1, \ldots, X_n)$ to indicate that its free variables are among $X_1, \ldots, X_n$.
Given an $\La$-structure $K$ and a tuple $(a_1, \ldots, a_n) \in K^n$, we can then simply write $\varphi(a_1, \ldots, a_n)$ instead of $\varphi(X_1 \subs a_1, \ldots, X_n \subs a_n)$.
As usual, for a sentence $\varphi$, we write $K \models \varphi$ to say that $\varphi$ holds in $K$.

We will primarily consider existential $\La$-formulas.
Following \autocite{DDF}, for $m \in \nat$, we write $\exists_m$-$\La$-formula for ``existential $\La$-formula with $m$ quantifiers'', i.e.~a formula which is logically equivalent to a formula of the form $\exists X_1, \ldots X_m \psi$ for some quantifier-free $\La$-formula $\psi$.
Similarly, we write $\forall_m$-$\La$-formula for ``universal $\La$-formula with $m$ quantifiers'', i.e.~a formula which is logically equivalent to a formula of the form $\forall X_1, \ldots, X_m \psi$ for some quantifier-free $\La$-formula $\psi$.
Given $m_1, m_2 \in \nat$, an $\exists_{m_1}\forall_{m_2}$-$\La$-formula is a formula equivalent to one of the form $\exists X_1, \ldots, X_{m_1} \psi$ where $\psi$ is an $\forall_{m_2}$-$\La$-formula; similarly, one defines $\forall_{m_1}\exists_{m_2}$-$\La$-formulas.
\begin{defi}\label{D:existdef}
Let $K$ be a field, $m, n \in \nat$.
A set $D \subseteq K^n$ is called \emph{existentially definable with $m$ quantifiers} if it is definable in $K^n$ by an $\exists_m$-$\Lar(K)$-formula.
\end{defi}
This coincides with the definition hinted at in the introduction in all interesting cases:
\begin{prop}\label{P:existdef}
Let $K$ be a field, $m, n \in \nat^+$, and suppose $D \subseteq K^n$ is existentially definable with $m$ quantifiers.
Then there exist $r \in \nat$ and polynomials $f_1, \ldots, f_r \in K[X_1, \ldots, X_n, Y_1, \ldots, Y_m]$ such that
\begin{equation}\label{eq:pos-ex}
D = \lbrace x \in K^n \mid \exists y \in K^m : f_1(x, y) = \ldots = f_r(x, y) = 0 \rbrace.
\end{equation}
Furthermore, if $K$ is not algebraically closed, we may assume without loss of generality that $r = 1$ in \eqref{eq:pos-ex}.
\end{prop}
\begin{proof}
By \autocite[Corollary 4.12]{DDF} we have that $D$ is definable by a positive-existential $\Lar(K)$-formula with $m$ quantifiers, i.e.~ a formula which is logically equivalent to $\exists Y_1, \ldots, Y_m \psi$ for some $\Lar(K)$-formula $\psi$ built up from atomic $\Lar(K)$-formulas using only conjunctions and disjunctions (no negations).
By \autocite[Remark 3.4]{DDF} this implies that $D$ can be described as in \eqref{eq:pos-ex} for certain $r$ and $f_1, \ldots, f_r$, where one may choose $r = 1$ when $K$ is not algebraically closed.
\end{proof}
We further observe that nothing would be gained if we were to work in the signature of fields $\La_{\fie}$ instead of the signature of rings $\Lar$:
\begin{prop}\label{P:LringvsLfield}
Let $\varphi(x_1, \ldots, x_n)$ be a quantifier-free $\La_{\fie}$-formula.
There exists a quantifier-free $\Lar$-formula $\psi(x_1, \ldots, x_n)$ such that, for every field $K$ interpreted as an $\La_{\fie}$-structure in the natural way, and for all $(a_1, \ldots, a_n) \in K^n$, we have
$$ K \models \varphi(a_1, \ldots, a_n) \enspace \Leftrightarrow \enspace K \models \psi(a_1, \ldots, a_n).$$
\end{prop}
Said informally, this proposition states that one can ``clear denominators'' from a quantifier-free $\La_{\fie}$-formula to obtain a quantifier-free $\Lar$-formula.
For completeness, we provide a formal proof.
\begin{proof}[Proof of \Cref{P:LringvsLfield}]
Consider $n \in \nat$ and two $\Lar$-terms $t(x_1, \ldots, x_n)$ and $s(x_1, \ldots, x_n)$. We can find polynomials $f, g \in \zz[X_1, \ldots, X_n]$ such that $t^K(a) = f(a)$ and $s^K(a) = g(a)$ for all fields $K$ and $a \in K^n$. Let $d \in \nat$ and $f_0, \ldots, f_d \in \zz[X_2, \ldots, X_n]$ be such that $f = \sum_{i=0}^d X_1^i f_i(X_2, \ldots, X_n)$, and consider $$h = \sum_{i=0}^d X_1^{d-i} f_i(X_2, \ldots, X_n).$$
We see now that the formula $t(x_1 \subs x_1^{-1}) \doteq s$ is equivalent for all fields to
\begin{align*}
&(h(x_1, \ldots, x_n) \doteq x_1^d g(x_1, \ldots, x_n) \wedge \neg(x_1 \doteq 0)) \\
&\vee (x_1 \doteq 0 \wedge f(0, x_2, \ldots, x_n) \doteq g(0, x_2, \ldots, x_n)).
\end{align*}
By recursively applying this procedure to a quantifier-free $\La_\fie$-formula $\varphi$, one can get rid of all occurrences of ${}^{-1}$ and obtain an equivalent quantifier-free $\Lar$-formula.
\end{proof}
\begin{gev}\label{C:LringvsLfield}
Let $K$ be a field, $m, n \in \nat$.
If a subset $D \subseteq K^n$ is definable by an $\exists_m$-$\La_{\fie}(K)$-formula, then it is definable by an $\exists_m$-$\Lar(K)$-formula.
\end{gev}
\begin{proof}
This is immediate from \Cref{P:LringvsLfield}.
\end{proof}
We further observe that, if $D_1, D_2 \subseteq K^n$ are existentially definable with $m_1$ and $m_2$ quantifiers respectively, then $D_1 \cup D_2$ is existentially definable with $\max \lbrace m_1, m_2 \rbrace$ quantifiers.
On the other hand, in the same situation, $D_1 \cap D_2$ is naively definable with $m_1 + m_2$ quantifiers.
The following result says that, a.o. for global fields, we can do slightly better in the latter case as well.
\begin{stel}\label{cor:finGenOverPerf}
Let $K$ be a field which is finitely generated over a perfect subfield.
For any $m_1, m_2, n \in \nat$ with $m_1, m_2 \geq 1$ and $D_1, D_2 \subseteq K^n$ such that $D_1$ is $\exists_{m_1}$-$\Lar(K)$-definable and $D_2$ is $\exists_{m_2}$-$\Lar(K)$-definable, we have that $D_1 \cap D_2$ is $\exists_{m_1+m_2-1}$-$\Lar(K)$-definable.
\end{stel}
\begin{proof}
See \autocite[Theorem 1.4]{DDF}.
\end{proof}

Finally, we mention that we currently do not have many adequate techniques available to show that a given subset of a global field is \textit{not} $\exists_m$-$\Lar(K)$-definable for a given natural number $m$; see the discussion in \autocite[Section 8]{DDF}.
In particular, we do not have any example of an $\exists$-$\Lar(K)$-definable subset of a global field $K$ of which we can show that it is not $\exists_2$-$\Lar(K)$-definable.

On the other hand, some necessary criteria have been found for a subset of a global field $K$ to be $\exists_1$-$\Lar(K)$-definable.
If $K$ is an imperfect field of characteristic $p$ (e.g.~a global field of characteristic $p$) and $k \in \nat$, then the set of $p^k$-th powers $K\pow{p^k}$ is an $\exists_1$-$\Lar$-definable infinite proper subring of $K$.
If $K$ is a global field, one can show that these are the only $\exists_1$-$\Lar(K)$-definable infinite proper subrings of $K$:
\begin{stel}\label{P:NotE1}
Let $K$ be a global field, $R \subseteq K$ an infinite proper subring of $K$.
Then $K \setminus R$ is not $\exists_1$-$\Lar(K)$-definable in $K$.
If $\charac(K) = p > 0$, and $R$ is $\exists_1$-$\Lar(K)$-definable in $K$, then $R = K\pow{p^k}$ for some $k \in \nat$.
\end{stel}
\begin{proof}
Let $p = \charac(K)$.
Assume first that, if $p > 0$, then $R \not\subseteq K\pow{p}$; we will later see how to reduce to this case.

By \autocite[Corollary 8.5]{DDF} (in view of \autocite[Corollary 4.21]{DDF}), to show that $R$ and $K \setminus R$ are not $\exists_1$-$\Lar(K)$-definable in $K$, it suffices to show that $R$ and $K \setminus R$ are not thin subsets of $K$ (see \autocite[Definition 8.1]{DDF}).
We will use that, if $L/K$ is a finite separable field extension and $D \subseteq L$ is a thin subset of $L$, then $D \cap K$ is a thin subset of $K$ \autocite[Corollary 12.2.3]{Fri08}.

If $p = 0$, let $K_0 = \qq$.
Otherwise, fix a transcendental element $T \in R$ such that $K/\ff_p(T)$ is a separable finite field extension, and set $K_0 = \ff_p(T)$.
Let $R_0 = R \cap K_0$.
Since $R_0$ contains either $\zz$ or $\ff_p[T]$, it is not thin in $K_0$ \autocite[Remark 8.14]{DDF}, whereby $R$ is not thin in $K$.
This concludes the proof that $R$ is not $\exists_1$-$\Lar(K)$-definable in $K$ if $R \not\subseteq K\pow{p}$.

To show that $K\setminus R$ is not $\exists_1$-$\Lar(K)$-definable in $K$, we consider two cases.
For the first case, suppose that $R_0 = K_0$.
Then $R$ is a field, hence there exists $x \in K$ such that $xR \subseteq K\setminus R$.
Since $R$ is not thin in $K$, neither is $xR$, hence neither is $K \setminus R$.
In the second case, $R_0 \neq K_0$.
Then $(K_0\setminus R_0)^{-1}$ contains the maximal ideal of a discrete valuation on $K_0$, hence is not thin $K_0$, whereby $K_0\setminus R_0$ is not thin in $K_0$ and thus $K\setminus R$ is not thin in $K$.
We conclude that $K \setminus R$ is not $\exists_1$-$\Lar(K)$-definable in $K$ if $R \not\subseteq K\pow{p}$.

We now consider the case where $R \subseteq K\pow{p}$.
In this case, $R$ is thin in $K$, whence $K \setminus R$ is not thin in $K$, and hence $K \setminus R$ is not $\exists_1$-$\Lar(K)$-definable.
We further make the following observation: if $R$ would be $\exists_1$-$\Lar(K)$-definable in $K$, then it would also be $\exists_1$-$\Lar(K\pow{p})$-definable in $K\pow{p}$.
Indeed, by \Cref{P:existdef} there would exist $f \in K[X, Y]$ such that
\begin{align*}
R = \lbrace x \in K \mid \exists y \in K : f(x, y) = 0 \rbrace = \lbrace x \in K\pow{p} \mid \exists y \in K\pow{p} : f(x, y^{1/p})^p = 0 \rbrace.
\end{align*}
Since $f(X, Y^{1/p})^p \in K\pow{p}[X, Y]$, we obtain the desired $\exists_1$-$\Lar(K\pow{p})$-definability of $R$ in $K\pow{p}$.
Furthermore, unless $R = K\pow{p}$, we have that $R \cap K\pow{p}$ is an infinite proper subring of the global field $K\pow{p}$.
Applying this observation repeatedly, and using that $R \not\subseteq K\pow{p^k}$ for some $k \in \nat$, we may reduce to the case where $R \not\subseteq K\pow{p}$, which we covered before, and conclude that indeed $R$ is not $\exists_1$-$\Lar(K)$-definable in $K$.
\end{proof}

\section{Quaternion algebras over global and local fields}\label{sect:Equat}

We recall some basic facts regarding global fields and quaternion algebras over them; most of these are also contained in \autocite[Sections 3 and 4]{DaansGlobal}.

For a valuation $v$ on a field $K$, we denote by $\mc{O}_v$ the valuation ring of $v$, by $\mf{m}_v$ the unique maximal ideal of $\mc{O}_v$, and by $K_v$ the fraction field of the completion of $\mc{O}_v$.
We also call the pair $(K, v)$ a valued field.
Given $a \in \mc{O}_v$, we denote by $\ovl{a}^v$ the residue of $a$ modulo $\mf{m}_v$.
Similarly, for a polynomial $f \in \mc{O}_v[X_1, \ldots, X_n]$, we denote by $\ovl{f}^v$ the corresponding residue polynomial in $Kv[X_1, \ldots, X_n]$.
For a field $K$, we denote by $\mc{V}_K$ the set of $\zz$-valuations on $K$, i.e.~the set of valuations on $K$ with value group $\zz$.

Suppose now that $K$ is a global field.
In this case, a $\zz$-valuation on $K$ corresponds to what is often called a finite place.
Observe that for $x \in K^\times$ there exist only finitely many $v \in \mc{V}_K$ for which $v(x) \neq 0$ (or see e.g.~\cite[Theorem 33:1]{OMe00}).
For $v \in \mc{V}_K$, the field $K_v$ is a complete $\zz$-valued field with a finite residue field.
We call a complete $\zz$-valued field with finite residue field a \emph{local field}.
We will call a valuation $v$ on a field $K$ \emph{dyadic} if $v(2) > 0$ (equivalently, $\charac(Kv) = 2$), and \emph{non-dyadic} otherwise.

We mention two standard results from valuation theory for later use.
For a univariate polynomial $f$, we denote by $f'$ its formal derivative.
\begin{stel}[Hensel's Lemma]\label{T:Hensel}
Let $K$ be field endowed with a complete $\zz$-valuation $v$.
Let $f \in \mc{O}_v[X]$ be a polynomial, and let $a_0 \in \mc{O}_v$ be such that $v(f(a_0)) > 2v(f'(a_0))$.
Then there exists some $a \in \mc{O}_v$ with $f(a) = 0$ and $v(a_0 - a) > v(f'(a_0))$.
\end{stel}
\begin{proof}
See e.g.~\cite[Theorem 1.3.1]{Eng05}.
\end{proof}
More generally, we call a valuation $v$ on a field $K$ \emph{henselian} if it satisfies the conclusion of \Cref{T:Hensel}.
We refer to \cite[Chapter 4]{Eng05} for a discussion of the structure theory of valued fields.
The only henselian valuations appearing in this paper will be the complete $\zz$-valuations on local fields, but we will state some auxiliary results for general henselian valuations.
\begin{stel}[Weak Approximation Theorem]\label{T:WAT}
Let $K$ be a field, $n \in \nat$, and let $v_1, \ldots, v_n$ be pairwise different $\zz$-valuations on $K$.
For any $a_1, \ldots, a_n \in K$ and $\gamma \in \zz$, there exists an $x \in K$ with $v_i(x-a_i) > \gamma$ for all $i \in \lbrace 1, \ldots, n \rbrace$.
\end{stel}
\begin{proof}
See e.g.~\cite[Theorem 2.4.1]{Eng05}; the independency assumption mentioned there is automatically satisfied for pairwise different $\zz$-valuations.
\end{proof}

A field is called \emph{real} if it carries a field ordering, \emph{nonreal} otherwise.
For a global field $K$ there is a one-to-one correspondence between the set of field orderings on $K$ and the set of field embeddings of $K$ into $\rr$.
In particular, a global field is real if and only if it can be embedded into $\rr$.

A \emph{quaternion algebra} over a field $K$ is a 4-dimensional central simple $K$-algebra.
We call a quaternion algebra \emph{split} if it has zero divisors, \emph{non-split} otherwise.
Given a field extension $L/K$ and a quaternion algebra $Q$ over $K$, we have that $Q \otimes_K L$ is a quaternion algebra over $L$.
We say that $Q$ is \emph{split over $L$} (respectively \emph{non-split over $L$}) if $Q \otimes_K L$ is split (respectively non-split).

Given $a, b \in K$ with $b(1+4a) \neq 0$, we define the $4$-dimensional $K$-algebra $[a, b)_K = K \oplus Ku \oplus Kv \oplus Kuv$ with $u^2 - u = a$, $v^2 = b$ and $uv+vu = v$.
This is a $K$-quaternion algebra, and in fact every $K$-quaternion algebra is of this form for some $a$ and $b$ \autocite[Section IX.10]{Alb39}.
For a $K$-quaternion algebra $Q$, we denote by $\Trd$ and $\Nrd$ the \emph{reduced trace} and \emph{reduced norm} maps $Q \to K$ respectively; see \autocite[Section 8.5]{Sch85} for the definition and basic properties.

A quaternion algebra $Q$ over a global field $K$ is called \emph{nonreal} if $Q$ is split over every embedding of $K$ into $\rr$.
By definition, if $K$ cannot be embedded into $\rr$ (i.e.~$K$ is nonreal) then all quaternion algebras over $K$ are nonreal.

Let $Q$ be a quaternion algebra over a field $K$.
Define
$$ \Delta Q = \lbrace v \in \mc{V}_K \mid Q \text{ is non-split over } K_v \rbrace.$$

\begin{prop}\label{P:quatsplit}
Let $K$ be a local field.
For every quadratic field extension $L/K$ and any quaternion algebra $Q$ over $K$, $Q$ is split over $L$.
\end{prop}
\begin{proof}
See \autocite[Section 17.10]{Pie82}.
\end{proof}
\begin{prop}\label{P:quatNonSplitNecessary}
Let $K$ be a local field with $\zz$-valuation $v$.
Let $a, b \in K$ be such that $(1+4a)b \neq 0$ and $[a, b)_K$ is non-split over $K$.
Then $v(a) \leq 0$, and furthermore at least one of the following holds:
\begin{enumerate}[(a)]
\item $v(b)$ is odd,
\item $v(2) = 0$ and $v(1+4a)$ is odd,
\item $v(2) > 0$ and $v(a) < 0$.
\end{enumerate}
\end{prop}
\begin{proof}
This is a rephrasing of \autocite[Proposition 4.1]{DaansGlobal}.
\end{proof}
\begin{stel}[Albert-Brauer-Hasse-Noether Theorem and Hilbert Reciprocity]\label{T:ABHN}
Let $K$ be a global field and let $Q$ be a nonreal $K$-quaternion algebra.
Then $\lvert \Delta Q \rvert$ is even, and furthermore we have $\Delta Q = \emptyset$ if and only if $Q$ is split.
Conversely, given a subset $S \subseteq \mc{V}_K$ such that $\lvert S \rvert$ is even, there exists up to $K$-isomorphism a unique nonreal $K$-quaternion algebra $Q$ such that $\Delta Q = S$.
\end{stel}
\begin{proof}
See \autocite[Theorem 8.1.17]{Neu15}.
\end{proof}
\begin{prop}\label{P:Quatsplitbyquadratic}
Let $K$ be a field.
Let $a, b \in K$ be such that $(1+4a)b \neq 0$ and set $Q = [a, b)_K$.
Furthermore, let $c, d \in K$.
The following are equivalent.
\begin{enumerate}[(i)]
\item\label{it:Qsplitoverquadratic} $Q$ is split over the splitting field of $X^2 - cX + d$.
\item\label{it:Qtrdnrd} There exists $\alpha \in Q \setminus K$ such that $\Trd(\alpha) = c$ and $\Nrd(\alpha) = d$.
\item\label{it:Qxyz} There exist $x, y, z \in K$ with $2x-c, y$ and $z$ not all zero such that
$$ x^2 + x(c-2x) - a(c-2x)^2 - b(y^2 + yz - az^2) = d.$$
\end{enumerate}
\end{prop}
\begin{proof}
The equivalence between \eqref{it:Qtrdnrd} and \eqref{it:Qxyz} follows immediately from the formulas for reduced norm and trace given in \autocite[Section 3]{DaansGlobal}.

We now discuss the equivalence between \eqref{it:Qsplitoverquadratic} and \eqref{it:Qtrdnrd}.
If $Q$ is already itself split, then $Q \cong \mbb{M}_2(K)$, $\Trd$ coincides with the matrix trace, and $\Nrd$ with the matrix determinant (see again \autocite[Section 8.5]{Sch85}). Since there exist non-diagonal matrices in $\mbb{M}_2(K)$ with any prescribed trace and determinant, it follows that both \eqref{it:Qsplitoverquadratic} and \eqref{it:Qtrdnrd} are satisfied.

Assume from now on that $Q$ is non-split.
For $\alpha \in Q \setminus K$ we have by definition of reduced trace and norm that $\alpha^2 - \Trd(\alpha)\alpha + \Nrd(\alpha) = 0$.
If \eqref{it:Qtrdnrd} holds, then $K(\alpha)$ is thus the splitting field of $X^2 - cX + d$. Since $Q$ is split over its subfield $K(\alpha)$ (see e.g.~\autocite[Theorem 5.4]{Sch85}), we obtain \eqref{it:Qsplitoverquadratic}.

Conversely, assume that \eqref{it:Qsplitoverquadratic} holds.
Since $Q$ is non-split, the splitting field of $X^2- cX + d$ is a proper quadratic extension of $K$.
By \autocite[Theorem IV.27]{Alb39} the splitting field of $X^2 - cX + d$ embeds over $K$ into $Q$.
Denoting by $\alpha \in Q$ an element for which $\alpha^2 - c\alpha + d = 0$, we obtain that $\alpha \not\in K$, $\Trd(\alpha) = c$ and $\Nrd(\alpha) = d$, as desired.
\end{proof}

\section{Defining valuation rings, individually and uniformly}\label{sect:definingValuationsGlobal}
In this section, we will show that a subring $R$ of a global field $K$ which is a finite intersection of valuation rings of $K$, is $\exists_3$-$\Lar(K)$-definable in $K$ (\Cref{P:valE3}).
This implies that in fact $R^n$ is $\exists_3$-$\Lar(K)$-definable in $K^n$ for every natural number $n$, as we will see in \Cref{C:binaryFormValuations}.
Finally, at the end of this section, we recall a result on uniform existential definability of finite intersections of valuation rings (essentially due to Poonen and Koenigsmann), see \Cref{C:EdefinabilitySemilocalRingsGlobalFields}.

For a field $K$ and $a \in K$, denote by $\sep{K}{a}$ the splitting field of $X^2 - X - a$ over $K$.
In other words, $\sep{K}{a} = K$ if $X^2 - X - a$ has a root in $K$, otherwise $\sep{K}{a} \cong K[X]/(X^2 - X - a)$.

\begin{lem}\label{lem:valE3}
Let $K$ be a global field.
Let $S$ be a finite set of $\zz$-valuations on $K$, $Q$ a nonreal quaternion algebra over $K$ such that $S \subseteq \Delta Q$.
Let $\pi, a \in K^\times$ such that for all $v \in \Delta Q$ one has $v(\pi) = 1$, $v(a) \geq v(1+4a) = 0$, and $X^2 - X - a$ has a root over $K_v$ if and only if $v \in S$.
Then
\begin{equation}\label{eq:split}
\bigcap_{v \in S} \mathcal{O}_v = \lbrace 0 \rbrace \cup \lbrace x \in K \mid Q \text{ is split over } \sep{K}{a - (\pi x^2)^{-1}} \rbrace .
\end{equation}
\end{lem}
\begin{proof}
Consider $x \in K^\times$ and let $L = \sep{K}{a - (\pi x^2)^{-1}}$.
Since $Q$ is nonreal (and hence remains nonreal over $L$) it follows by \Cref{T:ABHN} that $Q$ is split over $L$ if and only if it is split over $L_w$ for all $\zz$-valuations $w$ on $L$.
Since for any $\zz$-valuation $w$ on $L$ we have that $L_w \cong LK_v = \sep{(K_v)}{a - (\pi x^2)^{-1}}$ for some $\zz$-valuation $v$ on $K$, we conclude that $Q$ is split over $L$ if and only if it is split over $\sep{(K_v)}{a - (\pi x^2)^{-1}}$ for all $v \in \Delta Q$.
In order to show \eqref{eq:split}, we thus have to show that $x \in \bigcap_{v \in S} \mc{O}_v$ if and only if $Q$ is split over $\sep{(K_v)}{a - (\pi x^2)^{-1}}$ for all $v \in \Delta Q$.
Finally, in view of \Cref{P:quatsplit}, for any $v \in \Delta Q$, we have that $Q$ is split over $\sep{(K_v)}{a - (\pi x^2)^{-1}}$ if and only if $\sep{(K_v)}{a - (\pi x^2)^{-1}}/K_v$ is a quadratic field extension, i.e.~if and only if $X^2 - X - (a - (\pi x^2)^{-1})$ is irreducible over $K_v$.
In summary, we are left to show the following:
$$ x \in \bigcap_{v \in S} \mc{O}_v \enspace\Leftrightarrow\enspace \forall v \in \Delta Q : X^2 - X - (a - (\pi x^2)^{-1}) \text{ is irreducible over } K_v. $$

Consider a valuation $v \in \Delta Q$.
Assume first that $x \in \mc{O}_v$.
Suppose that $\alpha \in K_v$ were a root of $X^2 - X - (a - (\pi x^2)^{-1})$.
Since we then must have $v(\alpha) < 0$, we compute that $2v(\alpha) = v(\alpha^2 - \alpha - a) = v((\pi x^2)^{-1}) = - 1 - 2v(x)$, which contradicts the fact that $v$ is a $\zz$-valuation.
We obtain that $X^2 - X - (a - (\pi x^2)^{-1})$ is irreducible over $K_v$.

On the other hand, for $v \in \Delta Q$ and $x \in K \setminus \mathcal{O}_v$ one has that $X^2 - X - (a - (\pi x^2)^{-1}) \equiv X^2 - X - a \bmod \mf{m}_v$, so by Hensel's Lemma (\Cref{T:Hensel}) we have that $X^2 - X - (a - (\pi x^2)^{-1})$ is has a root over $K_v$ if and only if $X^2 - X - a$ has a root over $K_v$, which by assumption is precisely the case when $v \in S$.

As desired, we conclude that for $x \in K^\times$ and $v \in \Delta Q$, we have that $X^2 - X - (a - (\pi x^2)^{-1})$ is irreducible over $K_v$ if and only if either $v \not\in S$ or $x \in \mc{O}_v$.
\end{proof}

\begin{prop}\label{P:valE3}
Let $K$ be a global field. Let $S$ be a finite set of $\zz$-valuations on $K$.
Then $\bigcap_{v \in S} \mathcal{O}_v$ has an $\exists_3$-$\Lar(K)$-definition in $K$.
\end{prop}
\begin{proof}
There exists a nonreal quaternion algebra $Q$ over $K$ such that $S \subseteq \Delta Q$, and furthermore, $\Delta Q$ is finite. This follows from the second part of \Cref{T:ABHN}, but can also be seen more elementarily, see e.g. \autocite[Lemma 6.3.6]{DaansThesis}.

By Weak Approximation (\Cref{T:WAT}), we can find $\pi, a \in K^\times$ such that the criteria of \Cref{lem:valE3} are satisfied and thus \eqref{eq:split} holds.
Thus it suffices to show that the set on the right in \eqref{eq:split} has an $\exists_3$-$\Lar(K)$-definition in $K$.
This is immediate from \Cref{P:Quatsplitbyquadratic}.
\end{proof}
\begin{opm}\label{R:valE}
The proof technique from \Cref{lem:valE3} and \Cref{P:valE3} goes back to Julia Robinson.
In fact, she showed that, for $K = \qq$, $\bigcap_{v \in S} \mc{O}_v$ is $\exists_3$-$\Lar$-definable with $S = \lbrace v_2, v_p \rbrace$ where $p$ is a prime with $p \equiv 3 \bmod 4$.
Similarly, she showed that $\bigcap_{v \in S} \mc{O}_v$ is $\exists_3$-$\Lar$-definable with $S = \lbrace v_p, v_q \rbrace$ where $p$ and $q$ are primes with $p \equiv 1 \bmod 4$ and such that $q$ is not a square modulo $p$ \autocite[Lemma 3 and 4]{Rob49}.
A similar argument can be found in \autocite[Lemma 3.1]{ZhangSunZinQ} for $S = \lbrace v_2 \rbrace$.

It is in any case well-known that in a global field, any valuation ring (and hence also any finite intersection of valuations rings) is existentially definable, see e.g.~\autocite[Proposition 3.1]{KimRoushRational} for number fields, \autocite[Lemma 3.22]{ShlapentokhGlobal} for global fields of odd characteristic, or \autocite[Theorem 5.15]{EisentragerThesis} for a proof covering all characteristics.
Our argument has the advantage of yielding in all cases a formula requiring only $3$ existential quantifiers.
\end{opm}

\begin{opm}\label{R:uniformity}
Inspection of the proof of \Cref{P:valE3} reveals that the defining formula is uniform in the following sense: there exists an $\exists_3$-$\Lar$-formula $\varphi(X, C_1, \ldots, C_m)$ such that, for every global field $K$ and every finite set $S$ of $\zz$-valuations on $K$, there exist parameters $c_1, \ldots, c_m \in K$ such that
$$ \bigcap_{v \in S} \mc{O}_v = \lbrace x \in K \mid K \models \varphi(x, c_1, \ldots, c_m) \rbrace.$$
An even more robust formula, but with more quantifiers, will be given in \Cref{C:EdefinabilitySemilocalRingsGlobalFields}.
\end{opm}

In the setting of \Cref{P:valE3} and with $S \neq \emptyset$, by \Cref{P:NotE1} we have that $\bigcap_{v \in S} \mc{O}_v$ is not $\exists_1$-$\Lar(K)$-definable in $K$.
\begin{ques}
Let $K$ be a global field. Let $S$ be a non-empty finite set of $\zz$-valuations on $K$.
Does $\bigcap_{v \in S} \mc{O}_v$ have an $\exists_2$-$\Lar(K)$-definition in $K$?
\end{ques}

When $R$ is a subring of a field $K$ and $R$ is existentially definable in $K$, then clearly also $R^{\times}$ is existentially definable in $K$, and $R^n$ is existentially definable in $K^n$ for all $n \in \nat$.
However, if for example $R$ is $\exists_m$-$\Lar(K)$-definable in $K$, then the naive way to existentially define $R^n$ in $K^n$ requires $nm$ quantifiers, or $n(m-1) + 1$ quantifiers if one can apply \Cref{cor:finGenOverPerf}.
We investigate cases in which a better bound on the number of required quantifiers can be found, in particular when $R$ is a finite intersection of valuation rings.
\begin{prop}\label{P:definingValuationUnits}
Let $R$ be an integrally closed domain and $K = \Frac(R)$.
For $x \in K$ one has
$$ x \in R^\times \text{ if and only if } x \neq 0 \text{ and } x + x^{-1} \in R.$$
In particular, if $R$ is $\exists_m$-$\Lar(K)$-definable for $m \in \nat$, then also $R^\times$ is $\exists_m$-$\Lar(K)$-definable.
\end{prop}
\begin{proof}
The implication from left to right is immediate. Conversely, assume that $x + x^{-1} \in R$, then $x \in R[x^{-1}]$.
This implies that $x$ is integral over $R$, and thus by assumption $x \in R$. Then also $x^{-1} = (x + x^{-1}) - x \in R$, and thus $x \in R^\times$.

The definability statement follows immediately.
\end{proof}
\begin{lem}\label{P:binaryFormValuations}
Let $K$ be a field, $n \in \nat$, $v$ a valuation on $K$. Let $f(X_1, \ldots, X_n) \in \mathcal{O}_v[X_1, \ldots, X_n]$ be a homogeneous polynomial such that $\ovl{f}^v \in Kv[X_1, \ldots, X_n]$ has no non-trivial zeros.
For any elements $a_1, \ldots, a_n \in K$ we have that
\begin{displaymath}
v(f(a_1, \ldots, a_n)) = \deg(f)\min \lbrace v(a_i) \mid i \in \lbrace 1, \ldots, n \rbrace \rbrace
\end{displaymath}
\end{lem}
\begin{proof}
If $a_1 = \ldots = a_n = 0$ there is nothing to show, so we may suppose that this is not the case.
The validity of the statement is not affected if $(a_1, \ldots, a_n)$ is scaled by an element of $K^\times$, so we may assume without loss of generality that $\min_{i=1}^n v(a_i) = 0$; we need to show that $v(f(a_1, \ldots, a_n)) = 0$. If not, then we would have $\ovl{f}^v(\overline{a_1}^v, \ldots, \overline{a_n}^v) = \overline{f(a_1, \ldots, a_n)}^v = 0$ in $Kv$, contradicting the assumption that $\ovl{f}^v$ has no non-trivial zeros.
\end{proof}
\begin{prop}\label{C:binaryFormValuations}
Let $K$ be a field and let $S$ be a finite set of valuations on $K$.
Let $R = \bigcap_{v \in S} \mc{O}_v$.
Suppose that $Kv$ is not algebraically closed for all $v \in S$.
For each $n \in \nat^+$, there exists a polynomial $G \in K[X_1, \ldots, X_n]$ such that, for all $x\in K^n$, we have 
$G(x) \in R$ if and only if $x \in R^n$.
In particular, if $R$ is $\exists_m$-$\Lar(K)$-definable for some $m \in \nat$, then also $R^n$ is $\exists_m$-$\Lar(K)$- definable.
\end{prop}
\begin{proof}
By replacing $S$ with an appropriate subset if necessary, we may assume that $\mc{O}_v \not\subseteq \mc{O}_w$ for any two distinct $v, w \in S$.

By the assumption on the residue fields and a version of Weak Approximation \autocite[Theorem 3.2.7.(3)]{Eng05}, we can find for each $v \in S$ a monic polynomial $f_v \in R[X]$ such that its residue $\overline{f_v}^v$ is of degree at least $2$ and irreducible over $Kv$.
Let $d = \prod_{v \in S} \deg(f_v)$ and $d_v = d/\deg(f_v)$ for each $v \in S$.
Denote by $f_v^\ast \in R[X, Y]$ the homogenisation of $f_v$, and observe that $\overline{f_v^\ast}^v$ has no non-trivial zeros over $Kv$.
Finally, again invoking \autocite[Theorem 3.2.7.(3)]{Eng05}, fix for each $v \in S$ an element $\alpha_v \in R$ such that $v(\alpha_v) = 0$ and $w(\alpha_v) > 0$ for all $w \in S \setminus \lbrace v \rbrace$.
We now define
$$ F(X, Y) = \sum_{v \in S} \alpha_v f_v^\ast(X, Y)^{d_v} \in R[X, Y],$$
which is homogeneous of degree $d$.
Consider $v \in S$.
We claim that for all $x, y \in K$ we have
$$ v(F(x, y)) = d\min \lbrace v(x), v(y) \rbrace $$
To see, this, note that by \Cref{P:binaryFormValuations} we have
$$ v(\alpha_v f_v^\ast(x, y)^{d_v}) = 0 + (\deg(f_v)d_v)\min \lbrace v(x), v(y) \rbrace = d\min \lbrace v(x), v(y) \rbrace, $$
whereas for $w \in S \setminus \lbrace v \rbrace$ we have
$$ v(\alpha_w f_w^\ast(x, y)^{d_w}) \geq v(\alpha_w) + (\deg(f_w)d_w)\min \lbrace v(x), v(y) \rbrace > d\min \lbrace v(x), v(y) \rbrace,$$
from which the desired statement follows.
Since this holds for all $v \in S$, we obtain that, for all $x, y \in K$, one has
$$ F(x, y) \in R \Leftrightarrow x \in R \text{ and } y \in R.$$

We can now inductively for $i \geq 1$ define polynomials $G_i(X_1, \ldots, X_i)$
by setting $G_1(X_1) = X_1$ and $G_i(X_1, \ldots, X_i) = F(G_{i-1}(X_1, \ldots, X_{i-1}), X_i)$.
We see that, for $x_1, \ldots, x_n \in K$, we have
$$ G_n(x_1, \ldots, x_n) \in R \Leftrightarrow x_1, \ldots, x_n \in R, $$
so $G$ is as desired.
The definability statement follows immediately.
\end{proof}

We conclude this section with a brief discussion of a uniform existential definability result essentially due to Poonen and Koenigsmann \autocite{Poo09, Koe16}, which will play a central role in the proof of the main theorem.
We recall from \autocite[Section 5]{DaansGlobal} the following definition.
For a field $K$ and a quaternion algebra $Q$ over $K$, we define the following subset of $K$:
\begin{displaymath}
S(Q) = \lbrace \Trd(\alpha) \mid \alpha \in Q \setminus K, \Nrd(x) = 1 \rbrace.
\end{displaymath}
\begin{stel}\label{T:Poonensets}
Let $Q$ be a nonreal quaternion algebra over a global field $K$.
Then
$$\bigcap_{v \in \Delta Q} \mc{O}_v = \lbrace x + y \mid x, y \in S(Q) \rbrace.$$
\end{stel}
\begin{proof}
See \autocite[Proposition 2.9]{Dit17}. 
In the case $K = \qq$, the argument goes back to \autocite[Proposition 6]{Koe16}, using ideas already developed in \autocite{Poo09}.
\end{proof}
\begin{prop}\label{C:EdefinabilitySemilocalRingsGlobalFields}
Let $K$ be a global field.
There exists an $\exists_6$-$\Lar(K)$-formula $\varphi(X, A, B)$ such that, for all $a, b \in K$ with $(1+4a)b \neq 0$ and such that $[a, b)_K$ is nonreal, we have
$$ \bigcap_{v \in \Delta [a, b)_K} \mc{O}_v = \lbrace x \in K \mid K \models \varphi(x, a, b) \rbrace.$$
\end{prop}
\begin{proof}
In view of \Cref{P:Quatsplitbyquadratic} we have that $$\lbrace (x, a, b) \in K^3 \mid (1+4a)b \neq 0 \text{ and } x \in S([a, b)_K) \rbrace$$
is $\exists_3$-$\Lar$-definable.
Furthermore, by \Cref{T:Poonensets}, we have for $a, b \in K$ with $(1+4a)b \neq 0$ and $[a, b)_K$ nonreal that
$$ x \in \bigcap_{v \in \Delta[a, b)_K} \mc{O}_v \enspace\Leftrightarrow\enspace \exists y \in K : y \in S([a, b)_K) \enspace\text{and}\enspace x - y \in S([a, b)_K). $$
By applying \Cref{cor:finGenOverPerf} with
\begin{align*}
D_1 &= \lbrace (x,y,a,b) \in K^4 \mid (1+4a)b \neq 0 \text{ and } y \in S([a, b)_K) \rbrace \quad\text{and} \\\quad D_2 &= \lbrace (x,y,a,b) \in K^4 \mid (1+4a)b \neq 0 \text{ and } x-y \in S([a, b)_K) \rbrace
\end{align*}
and using that $D_1$ and $D_2$ are both $\exists_3$-$\Lar(K)$-definable, we obtain the desired result.
\end{proof}

\section{Universally defining rings of $S$-integers}\label{sect:ZinQQO}
We now work our way towards the universal definability of rings of $S$-integers in global fields with $10$ quantifiers (\Cref{T:nicomainthmQO}).

\begin{lem}\label{P:EtoAgeneral}
Let $V$ be a non-empty set of valuations on a field $K$, $n \in \nat$.
The set $\bigcup_{v \in V} \mathfrak{m}_v$ has an $\exists_n$-$\Lar(K)$-definition in $K$ if and only if $\bigcap_{v \in V} \mathcal{O}_v$ has an $\forall_n$-$\Lar(K)$-definition in $K$.
\end{lem}
\begin{proof}
By \Cref{C:LringvsLfield} it suffices to show that $\bigcup_{v \in V} \mathfrak{m}_v$ has an $\exists_n$-$\La_{\fie}(K)$-definition in $K$ if and only if $\bigcap_{v \in V} \mathcal{O}_v$ has an $\forall_n$-$\La_{\fie}(K)$-definition in $K$.
This in turn follows from the observation
\begin{displaymath}
\bigcap_{v \in V} \mathcal{O}_v = \left(K \setminus \left(\bigcup_{v \in V} \mathfrak{m}_v\right)^{-1}\right) \cup \lbrace 0 \rbrace.
\end{displaymath}
\end{proof}

Following \autocite[Section 6]{DaansGlobal}, for a global field $K$, a non-empty finite set $S \subseteq \mc{V}_K$ and $u \in \bigcap_{v \in S} \mc{O}_v^\times$, define the set
$$ \Phi_u^S = \left\lbrace (a, b) \in K^2 \enspace\middle|\enspace b \in \bigcap_{v \in S}\mc{O}_v^\times, a \equiv u \bmod \prod_{v \in S} \mf{m}_v \right\rbrace.$$
\begin{lem}\label{L:XE3}
Let $K$ be a global field, $S \subseteq \mc{V}_K$ a non-empty finite set and $u \in \bigcap_{v \in S}\mc{O}_v^\times$.
The set $\Phi_u^S$ has an $\exists_3$-$\Lar(K)$-definition in $K^2$.
\end{lem}
\begin{proof}
By Weak Approximation, we can find $\pi \in K^\times$ with $v(\pi) = 1$ for all $v \in S$.
We see that for $a, b \in K$ we have that
\begin{align*}
(a, b) \in \Phi_u^S \enspace&\Leftrightarrow\enspace b \in \bigcap_{v \in S} \mathcal{O}_v^\times \text{ and } \frac{a-u}{\pi} \in \bigcap_{v \in S} \mathcal{O}_v \\
&\Leftrightarrow\enspace \frac{b^2 + 1}{b}, \frac{a-u}{\pi} \in \bigcap_{v \in S}\mathcal{O}_v.
\end{align*}
where the second equivalence follows from \Cref{P:definingValuationUnits}.
By \Cref{P:valE3} $\bigcap_{v \in S}\mathcal{O}_v$ is $\exists_3$-$\Lar(K)$-definable, and then the desired result follows from \Cref{C:binaryFormValuations} (and \Cref{C:LringvsLfield}).
\end{proof}
\begin{lem}\label{lem:gcharnot2}
Let $(K, v)$ be a valued field and consider the rational function
\begin{equation*}\label{eq:gcharnot2}
g(X, Y) = \frac{16X^4}{1+4X^2} - \left(\frac{(Y-1)^2}{Y}\right)^2 \in K(X, Y).
\end{equation*}
Let $a, b \in K$ with $(1+4a^2)b \neq 0$.
We have the following:
\begin{enumerate}
\item\label{it:gcn2-1} If $1+4a^2, b \in \mc{O}_v^\times$, then $g(a, b) \in \mc{O}_v$.
\item\label{it:gcn2-2} If $v(1+4a^2) = 0$ and $v(b) \neq 0$, then $v(g(a, b)) = -2\lvert v(b) \rvert$.
\item\label{it:gcn2-3} If $v$ is henselian and non-dyadic, $X^2 - X - a^2$ is irreducible, and $g(a, b) \in \mc{O}_v$, then $1+4a^2, b \in \mc{O}_v^\times$.
\end{enumerate}
\end{lem}
\begin{proof}
We can compute that for $a \in K$ we have
\begin{displaymath}
v\left(\frac{16a^4}{1+4a^2}\right) \begin{cases}
= -v(1+4a^2) < 0 &\text{ if } v(1+4a^2)>0, \\
= 2v(a) + 2v(2) < 0 &\text{ if } v(1+4a^2) < 0, \\
\geq 0 &\text{ if } v(1+4a^2) = 0,
\end{cases}
\end{displaymath}
and similarly, for $b \in K$
\begin{displaymath}
v\left(\frac{(b-1)^2}{b}\right)\begin{cases}
= -\lvert v(b) \rvert < 0 &\text{ if } v(b) \neq 0, \\
\geq 0 &\text{ if } v(b) = 0.
\end{cases}
\end{displaymath}
\eqref{it:gcn2-1} and \eqref{it:gcn2-2} now follow immediately.

For \eqref{it:gcn2-3}, assume that $v$ is henselian and non-dyadic, $X^2 - X - a^2$ is irreducible, and either $1+4a^2 \not\in \mc{O}_v^\times$ or $b \not\in \mc{O}_v^\times$; we need to show that $g(a, b) \not\in \mc{O}_v$.
If $b \in \mc{O}_v^\times$, then this is immediate from the computations in the above paragrapgh.
Assume for the sake of a contradiction that $b \not\in \mc{O}_v^\times$ and $v(g(a, b)) \geq 0$.
Then
$$ v\left(\left(\frac{4a^2b}{(b-1)^2}\right)^2\frac{1}{1+4a^2} - 1\right) = v\left(\left(\frac{b}{(b-1)^2}\right)^2g(a, b)\right) \geq \lvert v(b) \rvert > 0.$$
Using that $(K, v)$ is henselian and non-dyadic, this implies that $1+4a^2$ is a square in $K$, contradicting the assumption that $X^2 - X - a^2$ was irreducible.
\end{proof}
\begin{lem}\label{lem:gchar2}
Let $(K, v)$ be a valued field and consider the rational function
\begin{equation*}\label{eq:gchar2}
g(X, Y) = X^5\left(\left(\frac{(Y-1)^2}{Y}\right)^2 - \left(\frac{(Y-1)^2}{Y}\right) - X^2\right) \in K(X, Y).
\end{equation*}
Let $a, b \in K^\times$.
We have the following:
\begin{enumerate}
\item\label{it:gc2-1} If $a, b \in \mc{O}_v^\times$, then $g(a, b) \in \mc{O}_v$.
\item\label{it:gc2-2} If $v(a) = 0$ and $v(b) \neq 0$, then $v(g(a, b)) = -2\lvert v(b) \rvert$.
\item\label{it:gc2-3} If $\charac(K) = 2$, $v$ is henselian, $X^2 - X - a^2$ is irreducible, and $g(a, b) \in \mc{O}_v$, then $a, b \in \mc{O}_v^\times$.
\end{enumerate}
\end{lem}
\begin{proof}
For $b \in K$ we have
\begin{displaymath}
v\left(\frac{(b-1)^2}{b}\right)\begin{cases}
= -\lvert v(b) \rvert < 0 &\text{ if } v(b) \neq 0, \\
\geq 0 &\text{ if } v(b) = 0.
\end{cases}
\end{displaymath}
From this, \eqref{it:gc2-1} and \eqref{it:gc2-2} follow immediately.

For \eqref{it:gc2-3}, assume that $\charac(K) = 2$, $v$ is henselian, $X^2 - X - a^2$ is irreducible, and either $a \not\in \mc{O}_v^\times$ or $b \not\in \mc{O}_v^\times$; we need to show that $g(a, b) \not\in \mc{O}_v$.
Observe that anyway $v(a) \leq 0$; otherwise $X^2 - X - a^2$ would be reducible by the henselianity of $v$.
More, precisely, we have for any $y \in K$ that $v(y^2 - y - a^2) \leq -4v(a)$ since $v$ is henselian.
If $v(a) < 0$, we thus obtain that $v(g(a, b)) < v(a) < 0$.
On the other hand, if $v(a) = 0$ and $v(b) \neq 0$, we obtain that $v(g(a, b)) = -2\lvert v(b) \rvert < 0$ by \eqref{it:gc2-2}.
This concludes the proof of \eqref{it:gc2-3}.
\end{proof}
For a field $K$ and $c \in K^\times$, define the set
\begin{align*}
\Odd(c) &= \lbrace v \in \mc{V}_K \mid v(c) \text{ is odd} \rbrace.
\end{align*}
\begin{lem}\label{nicomainthm}
Let $K$ be a global field. Let $\pi \in K^\times$ be such that $S = \Odd(\pi)$ has odd cardinality. Let $u \in K^\times$ be such that for all $v\in S$ one has $v(u) = 0$ and $X^2 - X - u^2$ is irreducible over $Kv$.
If $\charac(K) = 2$, let $g(X, Y) \in K(X, Y)$ be as in \Cref{lem:gchar2}.
If $\charac(K) \neq 2$, then assume that $S$ contains all dyadic valuations, and let $g(X, Y) \in K(X, Y)$ be as in \Cref{lem:gcharnot2}.

For $x \in K$ we have
\begin{displaymath}
x \in \bigcup_{v \in \mc{V}_K \setminus S} \mf{m}_v \enspace\Leftrightarrow\enspace \exists (a, b) \in \Phi_u^S : \frac{a^2x^2g(a, b)}{1-x-a^2x^2} \in \bigcap_{v \in \Delta [a^2, b\pi)_K} \mc{O}_v.
\end{displaymath}
\end{lem}
\begin{proof}
We first consider the implication from left to right.
Consider $x \in \mf{m}_w$ for some $w \in \mc{V}_K \setminus S$.
As in the proof of \autocite[Lemma 6.6]{DaansGlobal}, we can find $(a, b) \in \Phi_u^S$ such that $\Delta [a^2, b\pi)_K = S \cup \lbrace w \rbrace$ and $w(1+4a^2) = 0$.
We must then have that $w(b\pi)$ is odd by \Cref{P:quatNonSplitNecessary}, and since $w(\pi)$ is even, this implies that $w(b)$ is odd.
After rescaling $b$ by a square in $K$ if necessary (which does not affect the $K$-isomorphism class of $[a^2, b\pi)_K$), we may assume without loss of generality that $w(b) = 1$.

By either \Cref{lem:gcharnot2} or \Cref{lem:gchar2} we obtain that $w(g(a, b)) = -2$, whereas $v(g(a, b)) \geq 0$ for $v \in S$.
Furthermore, since for all $v \in S \cup \lbrace w \rbrace$ one has that $X^2 - X - a^2$ is irreducible over $Kv$, and hence the form $X^2 - XY - a^2Y^2$ has no non-trivial zeroes over $Kv$, we compute by \Cref{P:binaryFormValuations} that for $v \in S \cup \lbrace w \rbrace = \Delta [a^2, b\pi)_K$ one has
\begin{align*}
v\left(\frac{a^2x^2g(a, b)}{1-x-a^2x^2}\right) &= 2v(a) + 2v(x) + v(g(a, b)) - \min \lbrace 0, 2v(a) + 2v(x) \rbrace \\
&= \max \lbrace v(g(a, b)), 2v(x) + v(g(a, b)) \rbrace \geq 0
\end{align*}
where the inequality in the end follows from the fact that $v(g(a, b)) \geq 0$ for $v \in S$, and from $w(g(a, b)) = -2$ and $w(x) \geq 1$.
We conclude that $\frac{a^2x^2g(a, b)}{1-x-a^2x^2} \in \bigcap_{v \in \Delta [a^2, b\pi)_K } \mc{O}_v$ as desired.

For the other implication, consider $(a, b) \in \Phi_u^S$ arbitrary.
As in the proof of \autocite[Lemma 6.6]{DaansGlobal} we see that $S \subseteq \Delta [a^2, b\pi)_K$ and that $[a^2, b\pi)_K$ is nonreal, so that by \Cref{T:ABHN} there exists $w \in \Delta [a^2, b\pi)_K \setminus S$.
By \Cref{P:quatNonSplitNecessary}, using that $w$ is non-dyadic if $\charac(K) \neq 2$, at least one of the following occurs:
\begin{enumerate}[(i)]
\item $w(b\pi)$ is odd.
Since $w(\pi)$ is even, this implies $w(b)$ is odd,
\item $\charac(K) \neq 2$ and $w(1+4a^2)$ is odd,
\item $\charac(K) = 2$ and $w(a) < 0$.
\end{enumerate}
Furthermore, we know that $X^2 - X - a^2$ is irreducible over $K_w$, since $[a^2, b\pi)_K$ is non-split over $K_w$.
It follows by \Cref{lem:gcharnot2} or \Cref{lem:gchar2} that $w(g(a, b)) < 0$.
We compute that for $x \in K$ with $\frac{a^2x^2g(a, b)}{1 - x - a^2x^2} \in \bigcap_{v \in \Delta[a^2, b)_K} \mc{O}_v$ we have
\begin{align*}
0 &\leq w\left(\frac{a^2x^2g(a, b)}{1 - x - a^2x^2}\right) \leq 2w(a) + 2w(x) + w(g(a, b)) - \min \lbrace 0, 2w(a) + 2w(x) \rbrace \\
&= \max \lbrace 2w(a) + 2w(x) + w(g(a, b)), w(g(a, b)) \rbrace
\end{align*}
Since $w(g(a, b)) < 0$, we infer that
$2w(x) \geq -2w(a) - w(g(a, b)) > 0$,
whereby $x \in \mf{m}_w$.
This shows the other implication.
\end{proof}
\begin{stel}\label{T:nicomainthmQO}
Let $K$ be a global field, $S \subseteq \mc{V}_K$ a non-empty finite set.
The set $\bigcap_{v \in \mathcal{V} \setminus S} \mathcal{O}_v$ has an $\forall_{10}$-$\Lar(K)$-definition in $K$.
\end{stel}
\begin{proof}
In view of \Cref{P:EtoAgeneral}, we only need to show that $\bigcup_{v \in \mathcal{V} \setminus S } \mathfrak{m}_v$ has an $\exists_{10}$-$\Lar(K)$-definition in $K$.
Furthermore, it suffices to show this for some finite set $S'$ of valuations containing the set $S$.
Indeed we have
$$
\bigcup_{v \in \mc{V} \setminus S} \mathfrak{m}_v = \bigcup_{v \in S'\setminus S} \mathfrak{m}_v \cup \bigcup_{v \in \mc{V} \setminus S'} \mathfrak{m}_v $$
and, for each $v \in S' \setminus S$ individually, $\mf{m}_v$ is $\exists_3$-$\Lar(K)$-definable by \Cref{P:valE3}: after fixing a uniformiser $\pi$ of $v$, one has $\mf{m}_v = \lbrace x \in K \mid x\pi^{-1} \in \mc{O}_v \rbrace$.
Since $S' \setminus S$ is finite, $\exists_{10}$-$\Lar(K)$-definability of $\bigcup_{v \in \mc{V} \setminus S} \mathfrak{m}_v$ thus follows from $\exists_{10}$-$\Lar(K)$-definability of $\bigcup_{v \in \mc{V} \setminus S'} \mathfrak{m}_v$.
As such, in the rest of the proof, we may without loss of generality replace $S$ by a larger finite set.

If $\charac(K) = 0$, we enlarge $S$ so that it contains all dyadic valuations. By \autocite[Lemma 6.7]{DaansGlobal} we may further enlarge $S$ so that $S = \Odd(\pi)$ for some $\pi \in K^\times$ and $\lvert S \rvert$ is odd.
Fix $u \in \bigcap_{v \in S} \mc{O}_v^\times$ such that $X^2 - X - u^2$ is irreducible over $Kv$ for all $v \in S$; such element $u$ exists by Weak Approximation and \autocite[Lemma 6.5]{DaansGlobal}.
By \Cref{nicomainthm} there is a rational function $g(X, Y) \in K(X, Y)$ such that, for any $x \in K$, one has
\begin{displaymath}
x \in \bigcup_{v \in \mc{V}_K \setminus S} \mf{m}_v \enspace\Leftrightarrow\enspace \exists a, b \in K :  (a, b) \in \Phi_u^S \text{ and } \frac{a^2x^2g(a, b)}{1-x-a^2x^2} \in \bigcap_{v \in \Delta [a^2, b\pi)_K} \mc{O}_v.
\end{displaymath}
Since $\Phi_u^S$ is $\exists_3$-$\Lar(K)$-definable by \Cref{L:XE3} and the sets $\bigcap_{v \in \Delta [a^2, b\pi)_K} \mc{O}_v$ are uniformly $\exists_6$-$\Lar(K)$-definable by \Cref{C:EdefinabilitySemilocalRingsGlobalFields}, we obtain that $\bigcap_{v \in \mc{V}_K \setminus S} \mf{m}_v$ is existentially definable with $2+3+6-1=10$ quantifiers by \Cref{cor:finGenOverPerf} (and in view of \Cref{P:LringvsLfield}).
\end{proof}
\begin{ques}\label{Q:erkZ}
What is the smallest natural number $m$ such that $\zz$ is $\forall_m$-$\Lar$-definable in $\qq$?
\end{ques}
By \Cref{T:nicomainthmQO} and \Cref{P:NotE1} we obtain that the answer to \Cref{Q:erkZ} is at least $2$ and at most $10$.

\section{Recursively enumerable subsets of $\qq$}\label{sect:recursive}
We conclude with a proof of the promised undecidability result concerning the $\forall_{9}\exists_{10}$-$\Lar$-theory of $\qq$ (\Cref{C:undecidable}).
We present the argument in a way that makes transparent how further quantitative improments to the universal definability of $\zz$ in $\qq$ would impact the undecidability result.
The argument is essentially a reformulation of the proof of \autocite[Theorem 1.3]{ZhangSunZinQ}.

To be precise: when we say that the $\forall_m\exists_n$-$\Lar$-theory of a ring $R$ is undecidable, we mean that there is no algorithm which takes as input an arbitrary $\forall_m\exists_n$-$\Lar$-sentence $\varphi$ and, after a finite amount of steps, outputs YES if $R \models \varphi$ and NO if $R \not\models \varphi$.

\begin{prop}\label{P:undecidable}
Let $m \in \nat$ such that $m \geq 4$.
Assume that $\zz$ is $\forall_{m}$-$\Lar$-definable in $\qq$.
Then every recursively enumerable subset of $\qq$ is $\exists_{10}\forall_{m}$-$\Lar$-definable in $\qq$.
Furthermore, every recursively enumerable subset of $\zz$ is $\exists_9\forall_m$-$\Lar$-definable in $\qq$.

In particular, the $\forall_9\exists_m$-$\Lar$-theory of $\qq$ is undecidable.
\end{prop}
\begin{proof}
Fix a polynomial $f \in \zz[X, Y]$ such that $f$ defines an injection $\zz \times \zz \to \nat$ (see e.g. \autocite[Lemma 8.19]{DDF}).
For a subset $A \subseteq \qq$, define
$$ \tilde{A} = \lbrace f(a, b) \mid a, b \in \zz, b \neq 0, \frac{a}{b} \in A \rbrace $$
and observe that for any $a \in \qq$ we have
$$ a \in A \enspace\Leftrightarrow\enspace \exists y_0 \in \qq (y_0 \in \zz, ay_0 \in \zz \text{ and } f(ay_0, y_0) \in \tilde{A}). $$
Now assume that $A$ is recursively enumerable.
Then also $\tilde{A}$ is recursively enumerable.
By \autocite[Theorem 1.1(i)]{Sun21} there exists a polynomial $g_{\tilde{A}} \in \zz[X, Y_1, \ldots, Y_9]$ such that
$$ \tilde{A} = \lbrace x \in \nat \mid \exists y_1, \ldots, y_8 \in \zz, y_9 \in \nat : g_{\tilde{A}}(x, y_1, \ldots, y_9) = 0 \rbrace. $$
We obtain that, for any $a \in \qq$, we have that $a \in A$ if and only if
\begin{equation}\label{eq:AE}
\exists y_0, \ldots, y_9 \in \qq \left( ay_0, y_0, \ldots, y_9 \in \zz, y_9 \geq 0, \text{ and } g_{\tilde{A}}(f(ay_0, y_0), y_1, \ldots, y_9) = 0 \right)
\end{equation}
Since $\zz$ is $\forall_m$-$\Lar$-definable in $\qq$ and the set of non-negative elements is $\forall_4$-$\Lar$-definable in $\qq$ by Euler's Four-Square Theorem, 
we obtain the desired $\exists_{10}\forall_{m}$-$\Lar$-definability of $A$ in $\qq$.
If $A \subseteq \zz$ then one may remove the quantification over $y_0$ in \eqref{eq:AE} and equivalently write 
\begin{equation}
\exists y_1, \ldots, y_9 \in \qq \left( y_1, \ldots, y_9 \in \zz, y_9 \leq 0, \text{ and } g_{\tilde{A}}(f(a, 1), y_1, \ldots, y_9) = 0 \right)
\end{equation}
to obtain that $A$ is $\exists_9\forall_m$-$\Lar$-definable in $\qq$.

For the final statement, fix a recursively enumerable subset $A$ of $\nat$ such that $\nat \setminus A$ is not recursively enumerable (in other words, $A$ is not recursive).
By the above, $A$ is $\exists_9\forall_m$-$\Lar$-definable in $\qq$.
But since $A$ is not recursive, there cannot be an algorithm which decides whether a given element of $\qq$ lies in $A$.
This shows that the $\exists_9\forall_{m}$-$\Lar$-theory - or, equivalently, the $\forall_9\exists_{m}$-$\Lar$-theory - of $\qq$ is undecidable.
\end{proof}
\begin{gev}\label{C:undecidable}
Every recursively enumerable subset of $\qq$ is $\exists_{10}\forall_{10}$-$\Lar$-definable in $\qq$.
Furthermore, every recursively enumerable subset of $\zz$ is $\exists_9\forall_{10}$-$\Lar$-definable in $\qq$.

In particular, the $\forall_{9}\exists_{10}$-$\Lar$-theory of $\qq$ is undecidable.
\end{gev}
\begin{proof}
This follows from \Cref{P:undecidable} and \Cref{T:nicomainthmQO}.
\end{proof}

\printbibliography
\end{document}